\documentclass[a4paper,12pt]{amsart}

\usepackage{amssymb}
\usepackage{amsmath,a4wide,inputenc,euscript}
\usepackage{amsthm}
\usepackage[english]{babel}
\usepackage{amsfonts}
\usepackage{graphics}
\usepackage{wasysym}
\usepackage[dvips]{epsfig}
\usepackage{graphicx}

\author{Nicolae Mihalache}
\address{KTH - Royal Institute of Technology \newline
         Department of Mathematics \newline
         100 44 Stockholm, Sweden \newline \newline
         Supported by the \emph{Knut and Alice Wallenberg Foundation \newline}}
\email{nicolae@kth.se}
\title{Julia and John revisited}
\date{}

\newtheorem*{koebe}{Koebe Theorem}
\newcounter{count}

\newtheorem{thm}[count]{Theorem}
\newtheorem{prop}[count]{Proposition}
\newtheorem{lem}[count]{Lemma}
\newtheorem{coro}[count]{Corollary}
\newtheorem*{rem}{Remark}
\theoremstyle{remark}

\newcommand{\N}{\mathbb{N}}

\newcommand{\R}{\mathbb{R}}
\newcommand{\C}{\mathbb{C}}
\newcommand{\al}{\alpha}
\newcommand{\be}{\beta}
\newcommand{\ga}{\gamma}
\newcommand{\de}{\delta}
\newcommand{\De}{\Delta}
\newcommand{\te}{\theta}
\newcommand{\om}{\omega}
\newcommand{\Om}{\Omega}
\newcommand{\si}{\sigma}
\newcommand{\la}{\lambda}

\def\ra{{\rightarrow}}
\def\se{{\subseteq}}
\def\ve{{\varepsilon}}

\def\ft{{\infty}}
\def\sm{{\setminus}}
\def\es{{\emptyset}}
\def\pa{{\partial}}
\def\RS{{\overline\C}}
\def\mo{^{-1}}

\def\ma{{\mbox{ and }}}

\newcommand\MO[1]{\mathop{\mathrm{#1}}\nolimits}

\newcommand\MOo[2]{\mathop{\mathrm{#1}}\nolimits\left(#2\right)}

\newcommand\Deg[2]{\deg_{#1}\left({#2}\right)}

\newcommand\D{\mathbb{D}}

\newcommand\SetEnu[2]{\ensuremath{\left\{{#1},\ldots,{#2}\right\}}}
\newcommand\SetDef[2]{\ensuremath{\left\{{#1}\ :\ {#2}\right\}}}

\def\SE{\SetEnu}

\newcommand\ol[1]{\ensuremath{\overline{#1}}}

\newcommand\sqno[1]{\ensuremath{\left({#1}_{n}\right)_{n\geq 1}}}
\newcommand\sqnz[1]{\ensuremath{\left({#1}_{n}\right)_{n\geq 0}}}

\newcommand\fd[2]{\ensuremath{:{#1}\rightarrow{#2}}}
\newcommand\limi[2]{\ensuremath{\lim\limits_{{#1}\ra\ft}{#2}}}

\newcommand\rthm[1]{{Theorem \ref{thm#1}}}

\newcommand\rpro[1]{{Proposition \ref{prop#1}}}

\newcommand\rlem[1]{{Lemma \ref{lem#1}}}

\newcommand\rcor[1]{{Corollary \ref{cor#1}}}

\newcommand\requ[1]{{(\ref{equ#1})}}

\newcommand\abs[1]{\ensuremath{\left|{#1}\right|}}

\newcommand\A{\ensuremath{\mathcal A}}

\def\dia{\MO{diam}}
\def\cmp{\MO{Comp}}
\def\dst{\MO{dist}}
\def\lift{\MO{lift}}
\def\deg{\MO{deg}}
\def\cri{\MO{Crit}}
\def\qhl{l_{qh}}
\def\qhd{\dst_{qh}}

\begin{document}

\begin{abstract}
We show that Fatou components of a semi-hyperbolic rational map are John domains and that the converse does not hold. This generalizes a famous result of Carleson, Jones and Yoccoz. 

We show that a connected Julia set is locally connected for a large class of non-uniformly hyperbolic rational maps. This class is more general than semi-hyperbolicity and includes Collet-Eckmann and Topological Collet-Eckmann maps and maps verifying a summability condition (as considered by Graczyk and Smirnov).
\end{abstract}

\maketitle

\section{Introduction}

Hyperbolic rational dynamics is very well understood and Julia sets of hyperbolic maps have good geometric and statisctial properties. Allow critical points in the Julia set and one may lose these good geometric and statistical properties. During the last two decades various classes of rational maps have been considered which display some form of non-uniform hyperbolicity. Such classes include sub-hyperbolic (\cite{CG}, \cite{J}), semi-hyperbolic (\cite{JJ}), Collet-Eckmann (\cite{CEH}),  Topological Collet-Eckmann (\cite{ETIC}), Recurrent Collet-Eckmann (\cite{MCN1}) and maps verifying a summability condition (\cite{NUH}). Maps from these classes retain some of the good geometric and statistical properties of the hyperbolic setting. The main result of Carleson, Jones and Yoccoz in \cite{JJ} states equivalence for \emph{polynomial} maps between a geometric condition, John regularity of the basin of infinity, and a topological condition on critical orbits, semi-hyperbolicity. 

In Theorem 1, we extend the result of Carleson \emph{et al} to the rational setting: semi-hyperbolicity implies John regularity for all components of the Fatou set; however, equivalence does not hold. In Theorem 2, we prove local connectivity of connected Julia sets for all of the classes of rational maps mentioned in the first paragraph.

Let $f$ be a rational map of degree at least $2$. We say that $f$ is \emph{semi-hyperbolic} if it has no parabolic cycles and all critical points in its Julia set $J$ are \emph{non-recurrent}. We say that $x$ is non-recurrent if $x\notin\om(x)$ where $\om(x)$ is the accumulation set of the orbit of $x$, 
$$\om(x)=\bigcap_{N\geq 0}\ol{\SetDef{f^n(x)}{n\geq 0}}.$$

A domain $\Om\se\RS$ is an $\ve$-\emph{John domain} if there is $z_0\in\Om$ such that for all $z_1\in\Om$ there exists an arc $\ga\se\Om$ connecting $z_1$ to $z_0$ and for all $z\in\ga$
$$\de(z)\geq\ve\de(z,z_1),$$
where $\de$ denotes the distance with respect to the spherical metric $\si$ and by $\de(z)$ we mean $\de(z,\pa\Om)$.

A closed set $A\se\RS$ is called \emph{locally connected} if for every $\tau>0$ there is $\te>0$ such that, for any two points $a,b\in A$ with $\de(a,b)<\te$, we can find a continuum B (i.e. compact connected set with more than one point) with
$$a,b\in B,\ \dia B<\tau.$$

As a consequence of the main result of \cite{MCN1}, semi-hyperbolic rational maps satisfy the \emph{Exponential Shrinking of components} condition $(ExpShrink)$. This property was proved for semi-hyperbolic polynomials in \cite{JJ}. By \cite{ETIC}, $ExpShrink$ is equivalent to \emph{Topological Collet-Eckmann} condition $(TCE)$ and to several standard conditions for non-uniform hyperbolicity in rational dynamics.

A rational map $f$ satisfies the \emph{Exponential Shrinking of components}
condition if there are $\lambda>1,\ r>0$ such that
for all $z\in J,\ n>0$ and every connected component $W$ of $f^{-n}(B(z,r))$
$$\MO{diam}{W}<\lambda^{-n}.$$

We introduce a weaker version of $ExpShrink$. We say that a rational map $f$ satisfies the \emph{Summable Shrinking of components} condition ($SumShrink$) if there are $r>0$ and $\sqno \om$ a sequence of positive numbers such that
$$\sum\limits_{n>0}\om_n<\ft,$$
and for all $z\in J,\ n>0$ and every connected component $W$ of $f^{-n}(B(z,r))$
$$\MO{diam}{W}<\om_n.$$

This property rules out the existence of rotation domains, Cremer points and parabolic cycles. Therefore, by the classification of periodic Fatou components (see Theorem IV.2.1 in \cite{CG}),  the Julia set of such a map which has no attracting cycles is the Riemann sphere. 

In this paper we prove the following facts.

\begin{thm}
Fatou components of a rational semi-hyperbolic map are John domains. The converse does not hold.
% There exists a rational map whose Fatou components are John domains but which is not semi-hyperbolic.
\label{thmMain}
\end{thm}

\begin{thm}
If the Julia set of an $SumShrink$ rational map is connected then it is locally connected.
\label{thmLC}
\end{thm}

We also prove in \rpro{Unif} that if the Julia set of a semi-hyperbolic rational map is connected then there exists $\ve>0$ such that each Fatou component is an $\ve$-John domain. This implies a stronger version of local connectivity, that is $\te$ depends linearly on $\tau$, using the same notation as in the definition. This is related to the Julia set being fractal as defined in \cite{JJ}: small balls centered on $J$ are pushed forward to the large scale with bounded degree. This gives control on the geometric distortion so $J$ resembles itself at any scale. Using that semi-hyperbolic rational maps satisfy $ExpShrink$, it is fairly easy to check that semi-hyperbolicity is equivalent to Julia set being fractal. See Theorem 2.1 in \cite{JJ} for the proof in the polynomial case.

In \cite{JJ}, the existence of the basin of attraction of infinity, which is super-attracting in the polynomial case, and properties of the hyperbolic metric are used to prove relations between the geometry of the Fatou set and the dynamics. John regularity can be better understood in full generality (for domains which are not  simply connected) using the quasi-hyperbolic metric as demonstrated in \cite{Deep}. In our construction we emulate features like equipotential curves and geodesic rays in an arbitrary attracting cycle of a rational map.

Let $\ga\se\Om$ be an arc, we define its quasi-hyperbolic length by
$$\qhl(\ga)=\int_\ga\frac{|d\si(z)|}{\de(z)}.$$
This induces the quasi-hyperbolic distance $\qhd(\cdot,\cdot)$ on $\Om$ by the standard construction.
Let also $l(\ga)$ define the length of $\ga$ with respect to the spherical metric.

The quasi-hyperbolic distance has been used to give an alternative definition of John domains in \cite{Deep}. It has also been extensively employed in \cite{CEH} and \cite{NUH} to study H\"older regularity (defined in the following section) and the more general integrable domains (defined in the last section).

In the polynomial case, local connectivity of connected Julia sets is easier to check. Assume $J$ is connected and let us denote by $A_\infty$ the basin of attraction of infinity. Then $A_\infty$ is simply connected, so if it is a John or even a H\"older domain, the Riemann mapping extends to a H\"older continuous map on $\ol D$. Therefore, by Carath\'eodory's theorem, $J=\pa\A_\ft$ is locally connected.

Every John domain is a H\"older domain and every H\"older domain is an integrable domain. Graczyk and Smirnov show in \cite{NUH} that every connected component of the boundary of an integrable domain is locally connected. Suppose that all Fatou components are integrable domains. \emph{A priori}, this does not imply that $J$ is locally connected even if it is connected, since in general there are infinitely many Fatou components.

\section{John regularity}

In this section we prove the aforementioned results. The first tool relates the quasi-hyperbolic metric to John regularity. It is in fact one implication of the main Theorem in \cite{Deep}. As the proof is reasonably short we include it here for completeness.

\begin{lem}
Let $\Om\se\RS$ be a domain, $z_0\in\Om$ and $M>0$. Suppose that for all $z_1\in\Om$ there exists an arc $\ga\se\Om$ connecting $z_1$ to $z_0$ such that for all (orientation preserving) arc $\ga'\se\ga$ connecting $w_1$ to $w_0$ with 
$$\qhl(\ga')\geq M$$
one has 
$$\de(w_1) \leq \frac12\de(w_0).$$
Then $\Om$ is a $\ve(M)$-John domain.
\label{lemDeep}
\end{lem}

\begin{proof}

Let $\ga$ be a concatenation of arcs $\ga_0\cdot\ga_1\cdot\ldots\cdot\ga_m$ with $\qhl(\ga_i)\leq M$ for $i=0,\ldots,m$.
Let $w_0=z_0,w_1,\ldots,w_m=z_1$ be their endpoints.
By hypothesis we may assume that for all $i=0,\ldots,m-1$
$$\de(w_i)=2^{-i}\de(w_0).$$

Let us denote $\de^+_i=\max\SetDef{\de(z)}{z\in\ga_i}$ and $\de^-_i=\min\SetDef{\de(z)}{z\in\ga_i}$. Then one may observe that
$$M\geq\qhl(\ga_i)\geq\int_{\de_i^-}^{\de_i^+}\frac{dx}x$$
and therefore
\begin{equation}
\de_i^+\leq e^M\de_i^-.
\label{equEM}
\end{equation}
As a consequence, for all $i=0,\ldots,m$
$$l(\ga_i)\leq\qhl(\ga_i)\de_i^+\leq Me^M\de(w_i)\leq Me^M2^{-i}\de(w_0)$$
so for all $z\in\ga_i$
\begin{equation}
\de(z,z_1)\leq 2^{-i}(2Me^M\de(w_0)).
\label{equD}
\end{equation}
Using inequality \requ{EM}, for all $z\in\ga_i$ and $i=0,\ldots,m$
$$\de(z)\geq e^{-M}\de(w_i)=e^{-M}2^{-i}\de(w_0),$$
which combined with inequality \requ{D} shows that for all $z\in\ga$
$$\de(z)\geq \frac{e^{-2M}}{2M}\de(z,z_1).$$

\end{proof}

H\"older regularity is more general than John regularity. In the particular case when the domain $\Om$ is simply connected, it is equivalent to that the Riemann mapping $\varphi\fd\D\Om$ can be extended to a H\"older continuous mapping on the closed unit disk, see Lemma 6 in \cite{CEH}. In this case $\pa\Om$ is locally connected by Carath\'eodory's theorem.

Let us write $A(\cdot)\apprle B(\cdot)$ whenever $A$ has at most order $O(B)$, that is there are constants $C_0>0$ and $C_1>0$ such that 
$$A(\cdot)\leq C_0B(\cdot) + C_1.$$
We also write $A(\cdot)\approx B(\cdot)$ when $A(\cdot)\apprle B(\cdot)$ and $B(\cdot)\apprle A(\cdot)$.

A domain $\Om\se\RS$ is a \emph{H\"older domain} if there is $z_0\in\Om$ such that for all $z\in\Om$ 
$$\qhd(z,z_0)\apprle -\log \de(z).$$

As a consequence of Proposition 3 in \cite{CEH}, the Main Theorem and the Complement to the Main Theorem (page 49) in \cite{ETIC} we obtain the following fact.

\begin{coro}
Let $f$ be a rational map of degree at least $2$. If $f$ satisfies $ExpShrink$ then all connected components of the Fatou set are H\"older domains. If $f$ has a fully invariant attractive Fatou component that is a H\"older domain, then $f$ satisfies $ExpShrink$.
\label{corESH}
\end{coro}

\begin{proof}[Proof of \rthm{Main}]

Let $f$ be semi-hyperbolic. Then by aforementioned results all periodic components of its Fatou set are attracting, as $f$ satisfies $ExpShrink$.
If $J=\RS$ then there is nothing to prove.
%As there are no parabolic cycles it is easy to show that there exists a neighborhood $U$ of $J$ which is backward invariant ($f^{-1}(U)\se U$), does not intersect critical orbits contained in the Fatou set and $\ol{U}$ does not intersect attracting periodic orbits. For an explicit construction see \cite{MCN1} (page 1274).

Let us first show that an attracting periodic Fatou component $\Om$ is a John domain. Without loss of generality we may assume that $f(\Om)=\Om$. Let
$p\in\Om$ with $f(p)=p$ and $|f'(p)|<1$. 
Then all orbits in $\Om$ are attracted by $p$, that is for all $z\in\Om$ 
$$\limi n{f^n(z)}=p.$$
All derivatives are spherical derivatives unless specified otherwise.

We construct a domain $V$ with $\ol V\se\Om$ such that $p\in V$, $\ol{f(V)}\se V$ and for all $z\in\Om\sm V$ with $f(z)\in V$ we have $f(z)\notin f(V)$.

For any open $W\se\Om$ that contains $p$, we define $n_W\fd\Om\N$ such that $n_W(z)$ is the smallest iterate of $z$ that enters $W$. As $\Om$ is the immediate basin of attraction of $p$, $n_W$ is well defined on $\Om$.

Let $W=B(p,r_0)$ for some $r_0>0$ such that $\pa W$ does not intersect critical orbits and $\ol{f(W)}\se W$. Therefore $\pa f^{-k}(W)$ is smooth and $\ol{f^{-k}(W)}\se f^{-(k+1)}(W)$ for all $k\geq 0$. For all $k\geq 0$, let $W_k=\cmp_pf^{-k}(W)$ be the connected component of $f^{-k}(W)$ that contains $p$. Remark that $\ol{W_k}\se\Om$ for all $k\geq 0$.
Let also 
$$\SE{p,p_1}{p_m}=f\mo(p)\cap\Om.$$
Then there are arcs $\ga_1,\ldots,\ga_m\se\Om$ connecting $p$ to $p_1,\ldots,p_m$. By compactness, there is $k_0\geq 0$ such that for all $i=1,\ldots,m$
$$\ga_i\se W_{k_0}.$$

Then for all $k\geq k_0$, $W_k$ has the following properties
\begin{enumerate}
\item $\ol{f(W_k)}\se W_k$,
\item $f\mo(W_k)\cap\Om=W_{k+1}$.
\end{enumerate}
Indeed, $f\mo(W_k)\cap\Om$ is connected for all $k\geq k_0$. Otherwise it would contain a preimage of $p$ in $\Om$ outside $W_{k_0}\se W_{k+1}$.

Let $\la>1$, $r>0$ be provided by the $ExpShrink$ property of $f$. As there are no parabolic cycles, critical orbits in the Fatou set do not accumulate on the Julia set. By eventually shrinking $r$ we may assume that for all $z\in J$, $n\geq 0$ and $U$ a component of $f^{-n}(B(z,r))$
$$U\cap\cri\se J,$$
where $\cri$ is the set of critical points of $f$. As the critical orbits in the Julia set are not recurrent, we may also assume that there exists $\mu\geq 1$ such that

\begin{equation}
\deg_U f^n\leq \mu.
\label{equDeg}
\end{equation}

As $f$ is locally holomorphic, we may assume that the diameter of any such pullback $U$ is sufficiently small so that, by induction, it is simply connected.

By compactness there is $k_1\geq k_0$ such that $\pa W_{k_1}$ is contained in a $r/100$ neighborhood of $\pa\Om\se J$. We set
$$V=W_{k_1+1}.$$
Let $V_n=f^{-n}(V)\cap\Om=W_{k_1+n+1}$ and $n(z)=n_V(z)$ for all points $z\in\Om$. If $n(z)>0$ for some $z\in\Om$ then
$$f^{n(z)}\in V\sm f(V),$$
thus for all $k>0$
$$n^{-1}(k)=V_k\sm V_{k-1}.$$

Let us state a simplified version of a classical distortion control tool, the Koebe Theorem. As derivatives and distances are expressed with respect to the spherical metric we add a condition on the diameter of the image of the unit disk.

\begin{koebe}
There exists $\kappa>0$ and for all $D>1$ there is $\rho>0$ such that if  $g\fd\D\D$ is univalent then
$$B\left(g(0),\kappa |g'(0)|\right)\se g(\D)$$
and for all $z\in B(0,\rho)$
$$D\mo\leq \left|\frac{g'(z)}{g'(0)}\right|\leq D.$$
\end{koebe}

For more general statements of this theorem, see Theorem 1.3 and Corollary 1.4 in \cite{POM} or Lemma 2.5 in \cite{BvS}.

The following lemma will be used in the sequel together with Koebe's Theorem and is a direct consequence of the Monodromy Theorem. 
\begin{lem}
Let $U$ be a simply connected open, $g$ a rational map and $U'$ a connected component of $g\mo(U)$. If $g$ has no critical points in $U'$ then it is univalent on $U'$ and $U'$ is simply connected.
\label{lemSC}
\end{lem}

Let us prove that the quasi-hyperbolic length of arcs outside $V$ is not increased (except for an uniform constant) by pullbacks.
\begin{lem}
There exists an universal constant $K>1$ such that if $\ga\se\Om\sm f(V)$ is an arc and $\ga'\se\Om$ a homeomorphic pullback of $\ga$, that is $\ga'$ is a connected component of $f^{-k}(\ga)$ for some $k>0$, then
$$\qhl(\ga')\leq K\ \qhl(\ga).$$
\label{lemQHL}
\end{lem}

\begin{proof}
Critical orbits  outside $J$ do not approach $J$ closer than $r$ and for all $z\in\Om\sm f(V)$, $\de(z)<r/100$. Therefore if $z\in\ga$ then the (local) branch of $f^{-k}$ that sends $\ga$ to $\ga'$ is univalent on $B(z,\de(z))$. Indeed $B(z,\de(z))\se B(x,r)$ for some $x\in J$ so all preimages of $B(z,\de(z))$ are simply connected by \rlem{SC}. Univalence also follows by \rlem{SC}. Koebe Theorem applied to this local branch shows that the statement holds locally. Lemma follows by compactness of $\ga$.
\end{proof}

For all $z\in\Om$ we construct an arc $\ga_z\se\Om$ without self-intersections that connects $z$ to $p$ and avoids critical orbits. By compactness, there exists $L>0$ such that for all $z\in\ol V$ there is such an arc $\ga_z\se\ol V$ that connects $z$ to $p$ with 
$$\qhl(\ga_z)\leq L,$$
and such that $\ga'_z=\ga_z\sm f(V)$ has exactly one connected component.
Let $z\in\Om\sm\ol V$ and $m=n(z)$ except if $z\in\pa{V_{n(z)-1}}$ when $m=n(z)-1$. Let $y=f^m(z)$ and $\ga'_z=f^{-m}(\ga'_y)$ connect $z$ to $z'\in\pa V_{m-1}$. We define inductively $\ga_z$ as the concatenation
$$\ga_z=\ga'_z\cdot\ga_{z'}.$$

Using \rlem{QHL} and that $\ol V\se V_1$ we conclude that for all $z\in\Om$
\begin{equation}
\qhl(\ga_z)\apprle n(z).
\label{equQHLN}
\end{equation}

Let $z\in\Om\sm V$. Then $y=f^{n(z)}(z)\in V\sm f(V)$ therefore $\de(y)<r/4$. Using $ExpShrink$ we obtain that
$$\de(z)\leq \la^{-n(z)}.$$
One may also remark that
$$\de(\pa V)\cdot ||f'||_\ft^{-n(z)}\leq\de(z).$$
As a consequence of these inequalities we conclude that for all $z\in\Om$
\begin{equation}
-\log\de(z)\approx n(z).
\label{equDEZ}
\end{equation}

\begin{rem}
Relations \requ{QHLN} and \requ{DEZ} show that $\Om$ is a H\"older domain. This is an alternative proof of the direct implication of \rcor{ESH} as we do not use bound \requ{Deg}. With a similar construction, a stronger version of relation \requ{QHLN} and an estimate of $\de(z)$ that implies relation \requ{DEZ} have been proved in Lemma 7 in \cite{CEH}.
\end{rem}

Let us denote $\ga_z(k)=\ga_z\cap\ol{V_k}\sm V_{k-1}$, $\ga_z^k=\ga_z\sm V_{k-1}$ and $\lift(z,k)=\ga_z\cap\pa  V_{k-1}$ for all $k=1,\ldots,n(z)$. 
Using the last relation and \rlem{QHL} there exists $A>1$ such that for all $z\in\Om\sm V$ and $0<k<n(z)$
$$A^k\cdot l(\ga_z(k))\apprle \int_{\ga_z(k)}\frac {|d\xi|}{\de(\xi)}=\qhl(\ga_z(k))\leq K\cdot L,$$
therefore by summation
\begin{equation}
l(\ga_z^k)\apprle A^{-k}.
\label{equShort}
\end{equation}
We may therefore find $n_0>0$ such that for all $z\in\Om\sm\ol{n\mo(n_0)}$
$$l(\ga_z^{n_0})\leq\frac r{100}.$$

For $z\in\Om$ and $z'\in\ga_z$ we denote by $\ga_z^{z'}$ the arc $\ga'\se\ga$ that connects (or lifts) $z$ to $z'$.
By compactness and using relations \requ{DEZ} and \requ{QHLN} for all $\eta>0$ there exists $M>0$ such that if $n(z')\leq n_0$ then
\begin{equation}
\qhl(\ga_z^{z'})\geq M\Rightarrow \de(z)\leq \eta\cdot\de(z').
\label{equM}
\end{equation}

Let $\ga_w^{w'}\se\Om\sm V_{n_0}$ with $\qhl(\ga_w^{w'})\geq K\cdot M$ where $K$ is provided by \rlem{QHL}. We show that if $\eta$ is sufficiently small then
\begin{equation}
\de(w)\leq\frac12\de(w').
\label{equHalf}
\end{equation}
By \rlem{Deep} this means that $\Om$ is a John domain. 

The following statements are Lemmas 3 and 5 in \cite{MCN1}.
\begin{lem}
Let $g$ be a rational map, $z\in\overline{\mathbb C}$ and $0<r<R<1$. Let $W=B(z,R)^{-1}$
and $W'=B(z,r)^{-1}$ with $W'\subseteq W$ and $\MO{diam} W<1$. If $\Deg Wg
\leq\mu$ then
$$\frac{\MO{diam}{W'}}{\MO{diam}{W}} < 64\left(\frac rR\right)^{\frac 1\mu},$$
where $B\mo$ denotes a connected component of $g\mo(B)$.
\label{lemMod}
\end{lem}

If $A$ is an annulus and $C_1,C_2$ are the connected components of $\overline{\mathbb C}\setminus {\overline A}$ then we denote
$$\MOo{dist}{\RS\sm A}=\MOo{dist}{C_1,C_2}.$$
\label{HD}
Let us also denote 
$$\dst(\pa A)=\inf\SetDef{r>0}{\pa C_1\se\pa C_2 + B(0,r)\mbox{ and }\pa C_2\se\pa C_1 + B(0,r)},$$
the Hausdorff distance between the two components of the boundary of $A$. Let us also remark that 
\begin{equation}
\dst(\RS\sm A)\leq \dst(\pa A),
\label{equHD}
\end{equation}
with equality only when $A$ is a round annulus.
\begin{lem}
Let $A\subseteq\overline{\mathbb C}$ be an annulus and $C_1,C_2$ the
components of $\overline{\mathbb C}\setminus {\overline A}$.
For each $\alpha>0$ there exists
% an increasing map $\alpha\rightarrow$
$\delta_\alpha > 0$ that depends only on $\alpha$ such that if $\MO{mod} A\geq\alpha$ then
$$\MOo{dist}{\overline{\mathbb C}\setminus A}\geq\delta_\alpha \min(\MO{diam} C_1,\MO{diam} C_2).$$
\label{lemThk}
\end{lem}

Let $m=n(w')-n_0$ so $n(z')=n(f^m(w'))=n_0$. Let also $z=f^m(w)$ and $x,x'\in \pa\Om\se J$ with $\de(x,z)=\de(z)$ and $\de(x',z')=\de(z')$. By construction 
$$\ga_z^{z'}=f^m(\ga_w^{w'}),$$
and  $\de(z')<\frac r{100}$, $l(\ga_z^{z'})<\frac r{100}$. Moreover, by the choice of $\ga_w^{w'}$, \rlem{QHL} and inequality \requ{M}
$$\de(z)<\eta\cdot\de(z').$$

Let $U$ be the connected component of $f^{-m}(B(x,r))$ that contains $w$ and $w'$. Let also $y,y'\in U$ be preimages of $x$ and $x'$ respectively, by the same branch of $f^{-m}$ (i.e. connected by a homeomorphic pullback of the path $[x,z]\cdot\ga_z^{z'}\cdot[z',x']$ that contains $w$ and $w'$).
Let $B_0=B(z,\de(z))$, $B_1=B(z,r/8)$, $B_2=B(z',r/4)$, $B_3=B(z',r/2)$ and $U_0$, $U_1$, $U_2$, $U_3$ their respective pullbacks by $f^{-m}$ such that $w\in U_0\subset U_1\subset U_2\subset U_3\subset U$. By \rlem{Mod}
\begin{eqnarray*}
\de(w) & \leq & \dia U_0\\
       & \leq & 64\ \dia U_1\left(\frac{8\ \de(z)}{r}\right)^{\frac 1\mu}\\
       & \leq & 64\ \eta^{\frac 1\mu}\dia U_2\left(\frac{8\ \de(z')}{r}\right)^{\frac 1\mu}.
\end{eqnarray*}
Therefore
\begin{equation}
\de(w)< 64\ \eta^{\frac 1\mu}\dia U_2.
\label{equDew}
\end{equation}
as $8\ \de(z')< r$.

As $\MO{mod}(B_3\sm B_2)>\frac{\log 2}{2\pi}$, an application of Gr\"otzsch inequality on conformal pullbacks of subannuli of $B_3\sm B_2$ that separate $U_2$ from the complementary of $U_3$ shows that
$$\MO{mod}(U_3\sm U_2)>\frac{\log 2}{2\pi\mu}.$$ 
For an explicit construction one may check the proof of \rlem{Mod} in \cite{MCN1}. By \rlem{Thk} there exists $d>0$ that depends only on $\mu$ such that
$$B(w',d\cdot \dia U_2)\se U_3.$$

Let $D=B(0,r')$ for some $0<r'<1$. The spherical, Euclidean and hyperbolic metric $\rho_D$ on $D$ are (uniformly on $r'$) comparable on $B(0,r'/2)$. Therefore there exists $\be\in(0,1)$ that does not depend on $r'$ such that for all $0<\te\leq\be/2$
$$B(0,\be\te r')\se\SetDef{\zeta\in D}{\rho_D(0,\zeta)<\te}\se B(0,\be\mo\te r').$$

Let $D'=B(w',d\cdot \dia U_2)$ and 
$$\te=\frac{2\be\de(z')}{r},$$
which is bounded from below as $n(z')=n_0$.
Then 
$$B(w',\be\te d\cdot \dia U_2) \se \SetDef{\zeta\in D'}{\rho_{D'}(w',\zeta)<\te}\se\SetDef{\zeta\in U_3}{\rho_{U_3}(w',\zeta)<\te}$$ 
and by Schwarz Lemma
$$f^m(B(w',\be\te d\cdot \dia U_2)) \se \SetDef{\zeta\in B_3}{\rho_{B_3}(z',\zeta)<\te} \se B\left(z',\be\mo\te\frac r2\right)=B(z',\de(z')).$$
Therefore $\be\te d\cdot \dia U_2\leq \de(w')$ which combined with inequality \requ{Dew} and the lower bound for $\te$ show inequality \requ{Half}, provided $\eta$ is sufficiently small.

We have shown that each periodic Fatou component is a John domain.
There are only finitely many such components. As any other component is a pullback of a periodic one, it is enough to show that pullbacks of $\Om$ are John domains. Let $\Om'$ be such a component with $f^p(\Om')=\Om$ and $V'=f^{-p}(V_{n_0})\se\Om'$. We may recall that for all $z\in\Om$, $\ga_z$ avoids critical orbits. For $w\in\Om'$ let $z=f^p(w)$ and $\ga_w$ be the component of $f^{-p}(\ga_z)$ that contains $w$. It connects $w$ to a preimage of $z_0$ in $\Om'$. Paths $\ga_w^{w'}\se\Om'\sm V'$ are lifted to $\ga_z^{z'}\se\Om$ with $n(z')=n_0$.
Again by \rlem{QHL} and inequality \requ{Half} it follows that $\Om'$ is a John domain, as there are only finitely many preimages of $z_0$ in $\Om'$.

Let $\ga$ be a Jordan curve and $D>1$. We say that $\ga$ is a $D$-quasicircle if for all $x,y\in\ga$, the subarc $\ga'$ of $\ga$ of smaller diameter that joins $x$ and $y$ satisfies
$$\dia \ga'\leq D\dst(x,y).$$
Both components of the complementary of a quasicircle on $\RS$ are John domains.

Let us show that there exists a rational map whose Fatou components are John domains but which is not semi-hyperbolic. Corollary 4.4 in \cite{YZ} provides a degree $2$ rational map $g$ which has two fixed Siegel disks $\De^0$ and $\De^\ft$ with the following properties. $\pa\De^0$ and $\pa\De^\ft$ are disjoint quasicircles, each containing a critical point $c_0$ and $c_\ft$ respectively. 

$\pa\De^0$ and $\pa\De^\ft$ are forward invariant sets and by Theorem V.1.1 in \cite{CG}, the orbits of $c_0$ and $c_\ft$ are dense in $\pa\De^0$ and $\pa\De^\ft$ respectively, as $g$ has no other critical points. Therefore both critical points are recurrent. By Theorems III.2.2, III.2.3, IV.2.1 and V.1.1, all Fatou components are preimages of $\De^0$ or $\De^\ft$. By \rlem{SC}, all Fatou components are simply connected and univalent. It is not hard to check that a preimage of a quasidisk (component of the complementary of a quasicircle) by a rational map is a John domain. Therefore all Fatou components of $g$ are John domains but both critical points are recurrent, thus $g$ is not semi-hyperbolic. 

\end{proof}

%%%%%%%%%%%%%%%%%%%%%%%%%%%%%%%%%%% Corollary 1 %%%%%%%%%%%%%%%%%%%%%%%%%%%%%%%%%%%%%%%%

Let us show that if the Julia set of a semi-hyperbolic map is connected then Fatou components are John with a uniform constant. In the following section we use this result to show a stronger version of local connectivity of the Julia set.

\begin{prop}
Let $f$ be a semi-hyperbolic rational map with connected Julia set. There exists $\ve>0$ such that any Fatou component of $f$ is an $\ve$-John domain.
\label{propUnif}
\end{prop}

\begin{proof}

Let $U$ be a Fatou domain of $f$. We call $U$ \emph{multivalent} if $f$ is not univalent on $U$. As the Julia set is connected all Fatou components are simply connected therefore by \rlem{SC} there are only finitely many multivalent Fatou components. 

We show that univalent pullbacks of $\Om$ are uniformly John domains. The general case can be treated with minor modifications. Let us use the notations introduced in the previous proof and assume that $f^p\fd{\Om'}\Om$ is univalent. By the proof of the Main Theorem there exists $M>0$ that does not depend on the choice of $\Om'$ such that if $\qhl(\ga_w^{w'})\geq M$ and $\ga_z^{z'}=f^p(\ga_w^{w'})\se\Om\sm V_{n_0}$ then 
$$\de(w)\leq\frac 12\de(w').$$
Therefore the only obstacle to uniformity is related to $\qhl(\ga_w)$ and $\de(w)$ when $w\in V'\Leftrightarrow z\in V_{n_0}$. As $f^p$ is univalent on $\Om'$, \rlem{QHL} applies to $\ga_z$ for all $z\in\Om$. Therefore $\qhl(\ga_w)$ is uniformly bounded (independently of the choice of $\Om'$). To complete the proof we show that there is a bound $R>0$ that depends only on $\Om$ and $V_{n_0}$ such that for all $w,w'\in V'$
\begin{equation}
\frac{\de(w)}{\de(w')}\leq R.
\label{equR}
\end{equation}

Let $g\fd\Om{\Om'}$ be a univalent branch of $f^{-p}$. Let $\rho=\rho(2)$ provided by Koebe's Theorem. Let us cover $\ol {V_{n_0}}$ with $m$ balls $B(x_i,r_i)$ such that for all $i=1,\ldots,m$
$$r_i\leq \rho\ \de(x_i).$$
then for all $z,z'\in V_{n_0}$
$$\left|\frac{g'(z)}{g'(z')}\right|\leq 4^m.$$

If 
$$S=\sup\limits_{z,z'\in V_{n_0}}\frac{\de(z)}{\de(z')},$$
again by Koebe Theorem applied to $g$ and $g\mo$, we may define $R$ in inequality \requ{R} by
$$R=\kappa^{-2}4^m S.$$

\end{proof}

\section{Local connectivity}

%%%%%%%%%%%%%%%%%%%%%%%%%%%%%%%%%%% Corollary 2 %%%%%%%%%%%%%%%%%%%%%%%%%%%%%%%%%%%%%%%%

Let us show that connected Julia sets of semi-hyperbolic maps satisfy a slightly stronger version of local connectivity. The construction developed for this purpose is extended to prove \rthm{LC}.

\begin{prop}
If the Julia set of a semi-hyperbolic rational map is connected then it is locally connected. Moreover, there is $\ve>0$ such
that the Julia set satisfies the local connectivity definition with $\tau=\ve\mo\te$.
\label{propSLC}
\end{prop}

\begin{proof}
By \rpro{Unif} there is $\ve>0$ such that any Fatou component $U$ is a simply connected $\ve$-John domain. We use an alternative definition of simply connected John domains given by Theorem 4.4 in \cite{QHJ}. As we only need the easy part of this theorem, we include a proof here for completeness.

If $U$ is a simply connected domain we say that the segment $[a,b]$ is a \emph{crosscut} of $U$ if $[a,b]\cap\pa U=\{a,b\}$ and $[a,b]\se\ol U$. 
\begin{lem}
Let $U$ be a $\ve$-John simply connected domain, $[a,b]$ a crosscut of $U$ and $U_1$, $U_2$ the connected components of $U\sm[a,b]$. Then
$$\min(\dia U_1,\dia U_2)\leq \ve\mo\ \de(a,b).$$
\label{lemCross}
\end{lem}

\begin{proof}
Let $z_0$ be the base point of $U$ with respect to which it is an $\ve$-John domain. Let also $U'$ be the component of $U\sm[a,b]$ that does not contain $z_0$. Let $x,y\in U'$ and $\ga_x,\ga_y$ the paths that connect $z_0$ to $x$ and $y$ respectively, provided by the definition of John domains. Let $x'\in[a,b]\cap\ga_x$ and $y'\in[a,b]\cap\ga_y$. We may choose the order of $x',y'\in[a,b]$ such that 
$$\de(a,b)=\de(a,x') + \de(x',y')+\de(y',b).$$
As $\ve\de(x,x')\leq\de(x')\leq \de(a,x')$ and $\ve\de(y',y)\leq\de(y')\leq \de(y',b)$, the triangle inequality completes the proof.
\end{proof}

Let $\tau>0$ and $a,b\in J$ with $\de(a,b)<\ve\tau$. We build a continuum $C\se J$ that contains $a$ and $b$ with
\begin{equation}
\dia C<\tau,
\label{equCont}
\end{equation}
therefore $J$ is locally connected.

Let $E=\pa([a,b]\cap J)$ with respect to the topology of the real line and $B=[a,b]\sm E$. $E$ is compact with empty interior thus $B$ is a dense open in $[a,b]$. For any connected component $I$ of $B$, $\pa I\se J$ and $I\se J$ or $I\cap J=\es$. If $I\se J$ we define 
$$C(I)=\ol I.$$
If $I\cap J=\es$ then $\ol I$ is a crosscut of a Fatou component $U$. Let $D$ be the connected component of $U\sm [a,b]$ with smaller diameter. Then we define
$$C(I)=\pa D\sm I.$$
In this case $C(I)$ is the image of a round arc by Carath\'eodory's theorem, as $\pa U$ is locally connected.

In both cases $C(I)\se J$ is a continuum that contains $\pa I$ with
\begin{equation}
\dia C(I)\leq\ve\mo\ \dia I,
\label{equDiam}
\end{equation}
by \rlem{Cross}. 

Let $\sqnz I$ be a sequence that contains every connected component of $B$ exactly once.
Let
$$C'=E\cup\bigcup_{n\geq 0}C(I_n)$$
and
$$C=\ol{C'}\se J.$$

We show that $C'$ is connected therefore $C$ is a continuum. Suppose that there are two sets $A_1$ and $A_2$ with $A_1\cap C'\neq\es$, $A_2\cap C'\neq\es$, $\ol{A_1}\cap A_2=A_1\cap \ol{A_2}=\es$ and $C'\se A_1\cup A_2$. We may assume that $a\in A_1$ so $E\cap A_2\neq\es$, otherwise $C'\se A_1$. Let 
$$x=\inf (E\cap A_2),$$ 
where $[a,b]$ is identified to $[0,1]$ for readability reasons. Suppose that $x=\sup I_k$ for some $k\geq 0$. Then $\inf I_k\in A_1$ so $C(I_k)\se A_1$ as $C(I_k)$ is connected. But then $x\in A_1\cap\ol{A_2}$, a contradiction. Thus $x$ is an accumulation point of $E\cap A_1$. If $x=\inf I_k$ for some $k\geq 0$ then $C(I_k)\se A_2$ by the definition of $x$. But this contradicts $\ol{A_1}\cap A_2=\es$. Therefore $x$ is an accumulation point of $E\cap A_2$. But this yields again a contradiction as $x\in E\se C'\se A_1\cup A_2$. Therefore $C'$ is connected.

Let us show that 
$$\dia C'\leq \ve\mo\ \de(a,b)$$
which implies inequality \requ{Cont} thus completing the proof.
Let us remark that $\dia C'=\dia(C'\sm E)$ as $E=\pa B$. It is enough to show that if $x\in C(I_n)$ and $y\in C(I_m)$ for some $n,m\geq 0$ then $\de(x,y)\leq \ve\mo\ \de(a,b)$. Let $\{x_1,x_2\}=\pa I_n$ and $\{y_1,y_2\}=\pa I_m$. We may assume $a\leq x_1 < x_2\leq y_1<y_2\leq b$ as the case $n=m$ is trivial. By inequality \requ{Diam}
$$\de(x,x_2)\leq \ve\mo\ \de(x_1,x_2) \ma \de(y_1,y)\leq \ve\mo\ \de(y_1,y_2).$$
Conclusion is reached by the triangle inequality.

\end{proof}

%%%%%%%%%%%%%%%%%%%%%%%%%%%%%%%%%%% Proposition 3 %%%%%%%%%%%%%%%%%%%%%%%%%%%%%%%%%%%%%%%%

\begin{prop}
Let $K\se\RS$ be a continuum and $\sqnz U$ the sequence of connected components of its complementary $\RS\sm K$. If all $\pa U_n$ are locally connected and
$$\limi n{\dia U_n}=0$$
then $K$ is locally connected.
\label{propLC}
\end{prop}

\begin{proof}
For all $\tau>0$ there is $m>0$ such that for all $n> m$
$$\dia U_n<\frac\tau 3.$$
Let $\te>0$ such that for $n=0,\ldots,m$ and all $a,b\in\pa U_n$ with $\de(a,b)<\te$ there exists a continuum $B\se\pa U_n$ that contains $a$ and $b$ with
$$\dia B<\frac\tau 3.$$
Using the construction described in the proof of \rpro{SLC}, $\te$ satisfies the definition of local connectivity for $K$.
\end{proof}

%%%%%%%%%%%%%%%%%%%%%%%%%%%%%%%%%%% Corollary 4 %%%%%%%%%%%%%%%%%%%%%%%%%%%%%%%%%%%%%%%%
A domain regularity that is more general than H\"older regularity is considered in \cite{NUH}. A domain $\Om$ is called \emph{integrable} if there exists $z_0\in\Om$ and an integrable function $H\fd{\R_+}{\R_+}$,
$$\int_0^\ft H(r)dr<\ft,$$
such that $\Om$ satisfies the following \emph{quasi-hyperbolic boundary condition}. For all $z\in\Om$
$$\de(z)\leq H(\qhd(z,z_0)).$$ 
H\"older domains correspond to \emph{exponentially fast integrable} domains, that is with $H(r)=\exp(C-\ve r)$. However, John domains and H\"older domains cannot be distinguished by their integrability function $H$. An immediate consequence of Lemma 11.5 and Fact 11.1 in \cite{NUH} is that all connected components of the boundary of an integrable domain are locally connected. For any attracting periodic Fatou component of a rational map, integrability is characterized in terms of derivative growth on backward orbits inside the domain, see Lemma 11.1 in \cite{NUH}. 
Therefore, an immediate consequence of Koebe Theorem shows that $SumShrink$ implies local connectivity of components of the boundary of periodic Fatou components.
In the same paper (Theorem 11), Graczyk and Smirnov show that this holds for rational maps that satisfy a given \emph{summability condition}, a generalization of the Collet-Eckmann condition. This condition does not imply nor is it a consequence of $ExpShrink$. For more details, see the concluding section.

\begin{proof}[Proof of \rthm{LC}]
If $J=\RS$ there is nothing to prove, therefore we assume that the Fatou set is non-empty. As discussed in the introduction, by $SumShrink$, the Fatou set consists of finitely many attracting periodic components and their preimages. The Julia set is connected therefore Fatou components are simply connected. Therefore the boundary of periodic Fatou components are locally connected. As pullbacks of locally connected compacts by rational maps (iterates of $f$), all boundaries of Fatou components are locally connected.

By \rlem{SC}, there are only finitely many multivalent Fatou components. Using \rpro{LC}, it is enough to show that the diameters of univalent pullbacks of some Fatou component $U$ tend to $0$. Let
$$\varphi\fd\D U$$
be the Riemann mapping which extends continuously to $\ol\D$ by Carath\'eodory's theorem.
Let $A=U\sm\varphi(\ol{B(0,R)})$ be an annulus with $0<R<1$ such that
$$\dst(\pa A)<\frac r2,$$
where $r$ is given by $SumShrink$ and $\dst(\pa A)$ denotes the Hausdorff distance between the components of $\pa A$ (see definition on page \pageref{HD}).

Let $\sqnz U$ be a sequence of univalent pullbacks of $U=U_0$ such that $f(U_{n+1})=U_n$ for all $n\geq 0$. Let also $\sqno A$ be the corresponding pullbacks of $A$. Therefore for all $n>0$
$$\MO{mod} A=\MO{mod} A_n,$$
and using a cover of $A$ with balls of radius $r$ centered on $\pa U\se J$, by $SumShrink$
$$\limi n{\dst(\pa A_n)}=0.$$

Let $C_n$ and $K_n$ be the connected components of $\RS\sm A_n$ with $\dia C_n\leq\dia K_n$ for all $n>0$. Remark that
\begin{equation}
\dia (C_n\cup A_n)\leq \dia C_n + 2 \dst(\pa A_n).
\label{equDia}
\end{equation}
If $n$ is sufficiently large $\dst(\pa A_n)<1/4$ and using \rlem{Thk} $\dia C_n<1/2$. Then $K_n$ contains half the Riemann sphere. Therefore there is at most one (sufficiently large) $n$ such that $K_n\se U_n$. Therefore, for all but finitely many $n>0$
$$U_n=A_n\cup C_n.$$
By \rlem{Thk} and inequalities \requ{HD} and \requ{Dia}
$$\limi n{\dia U_n}=0,$$
which  completes the proof.
\end{proof}

\section{Further remarks}

In \cite{Deep}, the hypothesis of \rlem{Deep} is shown to be an equivalent definition of John regularity.
For simply connected domains, quasi-hyperbolic and hyperbolic metrics are comparable. In this case \rlem{Deep} has been used in \cite{JJ}, see also \cite{J}. In \cite{QHJ} it is proved that quasi-hyperbolic geodesics can replace arbitrary paths in the definition of the John regularity only in the simply connected case. 

A stronger version of John regularity, \emph{uniformly John} property, is considered in \cite{UJ}. In the case of simply connected domains it is equivalent to John regularity. Polynomials whose basin of infinity satisfy this property are characterized in terms of topological properties of critical orbits in the Julia set.

Graczyk and Smirnov proved in \cite{CEH} that Fatou components of a \emph{Collet-Eckmann} map (see definition below) are H\"{o}lder domains. The converse problem was considered by Przytycki in \cite{HCE}. It holds provided the orbit of each critical point in the Julia set of a polynomial does not accumulate on other critical points. The existence of a fully invariant Fatou component is essential, as is the case in \rcor{ESH}.

Relations between derivative growth and the geometry of Fatou components have also been studied in \cite{JCT}. All aforementioned regularity conditions are discussed in a systematic way.

In \cite{WHPOP} it is proved that polynomial derivative growth on repelling periodic orbits of a polynomial implies that the basin of attraction of infinity is an integrable domain. More precisely, it is required that the derivative on repelling periodic orbits of period $n$ is of order at least $n^{5+\ve}$. As a consequence, if the Julia set is connected then it is locally connected. This result has been improved in \cite{IWHPOP}, only growth of order $n^{3+\ve}$ is required in the rational case.
 
The assumption $J$ connected in \rpro{Unif} can be replaced by the condition that there are only finitely many multivalent Fatou components. If this condition fails then it is not hard to show that there are two critical points that are separated by infinitely many Fatou components. \emph{A priori}, this situation cannot be excluded. Similar phenomena may occur even for hyperbolic dynamics, see examples of dynamics in the last chapter of \cite{Bea}.

Suppose the Julia set is not connected, the components of the Fatou set are integrable domains, and their diameter tends to $0$. Then one may show that the connected components of the Julia set are locally connected. Only minor modifications in the proof of \rpro{LC} are needed.

Finally, let us define the summability condition as considered by Graczyk and Smirnov in \cite{NUH}. Let $f$ be a rational map of degree at least $2$, $J$ its Julia set and $\cri$ its critical set. For technical reasons we assume that critical orbits in the Julia set do not contain critical points but an additional construction overcomes this obstacle. Let 
$$\si_n:=\min\SetDef{\abs{(f^n)'(f(c))}}{c\in\cri\cap J}.$$
Suppose also that $f$ has no parabolic periodic points. We say that $f$ satisfies the \emph{summability condition} with exponent $\al$ if
$$\sum_{i=1}^\ft(\si_n)^{-\al}<\ft.$$
This condition generalizes the Collet-Eckmann condition which requires exponential growth of $\sqno\si$.
Let also $\mu_{max}$ be the maximal multiplicity of critical points in $J$.
Proposition 7.2 in \cite{NUH} shows that if $f$ satisfies the summability condition with exponent 
$$\al=\frac 1{1+\mu_{max}},$$
then $f$ satisfies $SumShrink$, so \rthm{LC} applies.

\bigskip
\textbf{Acknowledgments.} The author would like to thank Jacek Graczyk who suggested that a connected Julia set should be locally connected when the Fatou components are John domains. Interesting questions and observations of Juan Rivera-Letelier and Neil Dobbs helped improve the presentation of the paper. Juan Rivera-Letelier pointed out the construction of Yampolsky and Zakeri and a mistake in the very first version of this paper.

\end{document}